\documentclass{article}
\usepackage[english]{babel}
\usepackage[utf8]{inputenc}
\usepackage{amsmath}
\usepackage{amsfonts}
\usepackage{graphicx}
\usepackage[colorinlistoftodos]{todonotes}
\usepackage{blindtext}
\usepackage{amssymb}
\usepackage{algorithm}
\usepackage{amsthm}
\usepackage{bbm}
\usepackage{url}
\newtheorem{theorem}{Theorem}[section]
\newtheorem{corollary}{Corollary}[theorem]

\newtheorem{proposition}{Proposition}[section]
\theoremstyle{remark}
\newtheorem*{remark}{Remark}
\newtheorem{conjecture}{Conjecture}[section]
\theoremstyle{definition}

\title{Counting the number of group orbits by marrying the Burnside process with importance sampling}

\author{Persi Diaconis\thanks{Department of Mathematics and Department of Statistics, Stanford University}\and Chenyang Zhong\thanks{Department of Statistics, Columbia University}}

\date{\today}

\begin{document}
\maketitle
\makeatletter
\renewcommand{\@makefnmark}{}
\footnotetext{Authors are listed alphabetically}
\makeatother
\begin{abstract}
This paper introduces a novel and general algorithm for approximately counting the number of orbits under group actions. The method is based on combining the Burnside process and importance sampling. Specializing to unitriangular groups yields an efficient algorithm for estimating the number of conjugacy classes of such groups.
\end{abstract}

\section{Introduction}\label{Sect.1}

The problem of enumerating unlabeled objects has garnered significant interest across various fields including mathematics, statistics, and computer science. Such enumeration problems can be formulated in terms of P\'olya theory. Concretely, let $\mathcal{X}$ be a finite set and $G$ a finite group acting on $\mathcal{X}$. This group action splits $\mathcal{X}$ into orbits, which represent unlabeled objects. Understanding the orbits under group actions is a huge, unfocused problem with many interesting special cases. See the survey \cite{KellerOrbits}. P\'olya theory \cite{PR} gives an exact formula for counting the number of group orbits.

In practice, however, the computational complexity of exactly counting the number of group orbits based on P\'olya theory can be huge when the cardinalities of $\mathcal{X}$ and $G$ are large. This is connected to the computational P\'olya theory developed by Jerrum and Goldberg \cite{Gol, Jer, GJ2}. In particular, they showed that for many natural problem instances, exactly counting the number of orbits is \#P-complete and hence is hard. Therefore, we have to resort to \emph{approximately} counting the number of orbits.

This paper introduces a novel and general algorithm for approximately counting the number of group orbits. The algorithm is based on combining the Burnside process and importance sampling. The Burnside process, introduced by Jerrum \cite{Jer2}, is a Markov chain whose stationary distribution is given by the uniform distribution on the set of group orbits. Running the Burnside process for long enough time gives an approximately uniform sample from the set of orbits. Importance sampling \cite{CD} is a widely used technique for approximating the expectation of a function with respect to a probability distribution $\nu$ given samples from another distribution $\mu$. Section \ref{Sect.2} below gives background materials on the Burnside process and importance sampling. Besides these two ingredients, the algorithm also makes use of a nested sequence of sets and groups. 

A primary application of our general counting algorithm is to count the number of conjugacy classes of unitriangular groups. For $n\in \mathbb{N}^{*}=\{1,2,\cdots\}$ and $q$ a prime power (namely, $q=p^{l}$ for some prime number $p$ and $l\in\mathbb{N}^{*}$), the unitriangular group $U_n(\mathbb{F}_q)$ consists of $n\times n$ upper triangular matrices over the finite field $\mathbb{F}_q$ with ones on the diagonal. It is the Sylow $p$-subgroup of the general linear group $GL_n(\mathbb{F}_q)$. There has been significant interest in the number of conjugacy classes of $U_n(\mathbb{F}_q)$ (denoted by $k(U_n(\mathbb{F}_q))$ herafter) within the group theory and combinatorics communities. In \cite{Hig}, Higman made the following famous conjecture:
\begin{conjecture}
For every $n\in \mathbb{N}^{*}$, $k(U_n(\mathbb{F}_q))$ is a polynomial in $q$.
\end{conjecture}
\noindent See \cite{pak2015higman} for a report of the current status of this conjecture. Higman \cite{Hig} also showed the following bound on $k(U_n(\mathbb{F}_q))$:
\begin{equation}\label{Bound_Higman}
    q^{\frac{n^2}{12}(1+o_n(1))}\leq k(U_n(\mathbb{F}_q)) \leq q^{\frac{n^2}{4}(1+o_n(1))},
\end{equation}
where $o_n(1)$ denotes a function of $n$ that is independent of $q$ and approaches $0$ as $n \rightarrow \infty$. Improvements to Higman's bound were obtained in \cite{vera2003conjugacy,marberg2008constructing,soffer2016upper,soffer2016combinatorics}, and Soffer \cite{soffer2016upper, soffer2016combinatorics} conjectured that the lower bound in (\ref{Bound_Higman}) is sharp. See also \cite{DT,marberg2011combinatorial,pak2015higman} for further results related to conjugacy classes of $U_n(\mathbb{F}_q)$. 

In light of those developments, it is important to calculate $k(U_n(\mathbb{F}_q))$ for concrete values of $n$ and $q$. However, this remains a challenging task given the super-exponential growth of $k(U_n(\mathbb{F}_q))$ indicated by (\ref{Bound_Higman}). The currently best-known exact values of $k(U_n(\mathbb{F}_q))$ are only up to $n\leq 16$ \cite{pak2015higman,soffer2016combinatorics}. Our counting algorithm, when specialized to the case where $\mathcal{X}=G=U_n(\mathbb{F}_q)$ and $G$ acts on itself by conjugation, provides an efficient method to estimate $k(U_n(\mathbb{F}_q))$. It gives estimated values of $k(U_n(\mathbb{F}_q))$ for $q=2,3$ up to $n\leq 32$, and the estimated values for $n\leq 16$ match the known exact values closely (see Section \ref{Sect.4.4} for details). 

The rest of this paper is organized as follows. Section \ref{Sect.2} reviews the background on the Burnside process and importance sampling. The general orbit counting algorithm is introduced in Section \ref{Sect.3}. Section \ref{Sect.4} discusses the application of the algorithm to count the number of conjugacy classes of unitriangular groups. Section \ref{Sect.5} concludes with final remarks.

\section{Background}\label{Sect.2}

This section presents background materials on the Burnside process and importance sampling. 

\subsection{Burnside process}\label{Sect.2.1}

Let $\mathcal{X}$ be a finite set, and let $G$ be a finite group that acts on $\mathcal{X}$. For any $g\in G$ and $x\in\mathcal{X}$, define
\begin{equation*}
    \mathcal{X}_g:=\{y\in\mathcal{X}: y^g =y\}, \quad G_x:=\{h\in G: x^h=x\}.
\end{equation*}
The Burnside process, introduced by Jerrum \cite{Jer2}, is a Markov chain on $\mathcal{X}$ where each iteration can be described as follows:
\begin{itemize}
    \item From $x\in \mathcal{X}$, sample $g\in G_x$ uniformly at random.
    \item Given $g$, sample $y\in \mathcal{X}_g$ uniformly at random. Then move to the new state $y$. 
\end{itemize}
Thus the transition matrix of the Burnside process is given by
\begin{equation*}
    K(x,y)=\sum_{g\in   G_x\cap G_y }\frac{1}{|\mathcal{X}_g|}\frac{1}{|G_x|}, \quad \text{for all }x,y\in\mathcal{X}.
\end{equation*}
For any $x\in\mathcal{X}$, denote by $O(x)$ the orbit containing $x$. It can be checked that the Burnside process is an ergodic, reversible Markov chain with stationary distribution
\begin{equation*}
    \pi(x)= \frac{z^{-1}}{|O(x)|}, \quad \text{for all }x\in\mathcal{X}, \quad \text{with } z=\#\text{orbits}.
\end{equation*}

By Dynkin's criterion (see \cite{KS}), the Burnside process can be lumped onto orbits. The stationary distribution of the lumped chain is the uniform distribution on the set of orbits. Thus the Burnside process provides a method for obtaining approximately uniform samples from the set of orbits.

There has been much recent interest on the rates of convergence of the Burnside process. We refer to \cite{Jer2,GJ,AF,D,Chen,dittmer2019counting,DZ,R,Pag} for relevant results in the literature. 

\subsection{Importance sampling}\label{Sect.2.2}

Let $\mu$ and $\nu$ be two probability measures on a set $\mathcal{X}$ such that $\nu$ is absolutely continuous with respect to $\mu$, and let $f$ be a measurable real-valued function on $\mathcal{X}$. Suppose that independent random values $X_1,X_2, \cdots,X_n$ are drawn from $\mu$. Importance sampling aims to estimate the integral $I(f)=\int_{\mathcal{X}} fd\nu$ based on $X_1,X_2,\cdots,X_n$. 

Let $\phi$ be the Radon-Nikodym derivative of $\nu$ with respect to $\mu$. Importance sampling estimates $I(f)$ by
\begin{equation*}
    \widehat{I(f)}=\frac{1}{n}\sum_{i=1}^n f(X_i)\phi(X_i).
\end{equation*}
It can be shown that 
\begin{equation*}
    \mathbb{E}\big[\widehat{I(f)}\big]=I(f).
\end{equation*}
Therefore, when the sample size $n$ is sufficiently large, $\widehat{I(f)}$ gives a precise estimate of $I(f)$.

We refer the reader to \cite{HH,Sri,CMR,Liu} for detailed accounts of importance sampling. The reader is also referred to \cite{CD} for a recent work on the sample size required in importance sampling. For the counting algorithm introduced in Section \ref{Sect.3}, the random values $X_1,X_2,\cdots,X_n$ are samples from an ergodic Markov chain. Performance guarantees for importance sampling in this context are discussed in \cite{madras1999importance}; see also \cite{robert2004monte, Liu}.

\section{General algorithm for orbit counting}\label{Sect.3}

This section gives the general algorithm for approximately counting the number of orbits under group actions. Assume that a finite group $G$ acts on a finite set $\mathcal{X}$. Denote by $O_1,\cdots,O_{k(\mathcal{X},G)}$ the orbits under this group action, where $k(\mathcal{X},G)$ is the number of such orbits. The aim is to approximately count the number of orbits $k(\mathcal{X},G)$. The algorithm is an extension of the basic tools of computational complexity relating counting and approximate sampling \cite{broder1986hard,jerrum1986random,SJ}.

To proceed, make the following assumptions on $\mathcal{X}$ and $G$:
\begin{itemize}
    \item There is a sequence of finite sets $\{\mathcal{X}_i\}_{i=1}^N$ with $\mathcal{X}_N=\mathcal{X}$ and $|\mathcal{X}_1|=1$, together with a sequence of finite groups $\{G_i\}_{i=1}^N$ with $G_N=G$ and $G_1$ being the trivial group, such that the group $G_i$ acts on the set $\mathcal{X}_i$. Moreover, there is a sequence of maps $\{\phi_i\}_{i=1}^{N-1}$ such that $\phi_i$ is a surjection from $\mathcal{X}_{i+1}$ to $\mathcal{X}_i$.
    \item There is an efficient algorithm which for a given element $x\in \mathcal{X}_i$ returns the size of the stabilizer of $x$ in $G_i$, for every $1\leq i\leq N$.
\end{itemize}
The assumptions are easily seen to be satisfied in many natural examples (see Section \ref{Sect.4} for an example on unitriangular groups). Note that, while $\mathcal{X}_i \subseteq \mathcal{X}$ is usually assumed (where $1\leq i\leq N$), $\mathcal{X}_i$ can be more general and does not have to be a subset of $\mathcal{X}$. In practice, it is required that the ratios $\frac{|\mathcal{X}_{i+1}|}{|\mathcal{X}_i|}$ and $\frac{|G_{i+1}|}{|G_i|}$ are not too large for every $1\leq i\leq N-1$.

In what follows, for any $1\leq i\leq N$ and $x\in \mathcal{X}_i$, denote by $O_i(x)$ the orbit (under the group action of $G_i$) in $\mathcal{X}_i$ that contains $x$, and by $Stab_i(x)$ the stabilizer of $x$ in the group $G_i$. The number of orbits of $\mathcal{X}_i$ under the group action of $G_i$ is denoted by $k(\mathcal{X}_i,G_i)$.

The first idea of the orbit counting algorithm is to split the (possibly very large) quantity $k(\mathcal{X},G)$ into the product of $N-1$ ratios (note that $|\mathcal{X}_1|=1$):
\begin{equation}\label{EQ_1}
    k(\mathcal{X},G)=\prod_{i=1}^{N-1}\frac{k(\mathcal{X}_{i+1},G_{i+1})}{k(\mathcal{X}_i,G_i)}.
\end{equation}
In order to estimate $k(\mathcal{X},G)$ precisely, it suffices to estimate the ratios $\frac{k(\mathcal{X}_{i+1},G_{i+1})}{k(\mathcal{X}_i,G_i)}$ for every $1\leq i\leq N-1$. 

In order to estimate the individual ratios $\frac{k(\mathcal{X}_{i+1},G_{i+1})}{k(\mathcal{X}_i,G_i)}$ for $1\leq i\leq N-1$, we propose a method that combines the Burnside process with importance sampling. The importance sampling step is based on the following proposition:

\begin{proposition}\label{Le1}
Suppose that $1 \leq i\leq N-1$. Let $T_{i+1}$ be a random element of the set $\mathcal{X}_{i+1}$ such that
\begin{equation*}
    \mathbb{P}(T_{i+1}=x)  = \frac{1}{k(\mathcal{X}_{i+1},G_{i+1})}\cdot\frac{1}{|O_{i+1}(x)|}, \quad \text{for all } x\in \mathcal{X}_{i+1}.
\end{equation*}
Then 
\begin{equation*}
    \frac{|G_{i+1}|}{|G_i|}\cdot\mathbb{E}\bigg[\frac{|Stab_{i}(\phi_{i}(T_{i+1}))|}{|Stab_{i+1}(T_{i+1})||\phi_{i}^{-1}(\phi_i(T_{i+1}))|}\bigg]=\frac{k(\mathcal{X}_i,G_i)}{k(\mathcal{X}_{i+1},G_{i+1})}.
\end{equation*}
\end{proposition}
\begin{proof}
By the orbit-stabilizer theorem, for any $x\in \mathcal{X}_{i+1}$,
\begin{equation*}
    \frac{|G_{i+1}|}{|Stab_{i+1}(x)|}=|O_{i+1}(x)|,\quad  \frac{|G_{i}|}{|Stab_{i}(\phi_i(x))|}=|O_{i}(\phi_i(x))|.
\end{equation*}
Therefore,
\begin{eqnarray*}
&& \frac{|G_{i+1}|}{|G_i|}\cdot\mathbb{E}\bigg[\frac{|Stab_{i}(\phi_i(T_{i+1}))|}{|Stab_{i+1}(T_{i+1})||\phi_i^{-1}(\phi_i(T_{i+1}))|}\bigg]\\
&=& \frac{1}{k(\mathcal{X}_{i+1},G_{i+1})}\sum_{x\in \mathcal{X}_{i+1}}\frac{|Stab_i(\phi_i(x))||G_{i+1}|}{|G_i||Stab_{i+1}(x)||\phi_i^{-1}(\phi_i(x))|} \cdot \frac{1}{|O_{i+1}(x)|}\\
&=& \frac{1}{k(\mathcal{X}_{i+1},G_{i+1})}\sum_{x\in \mathcal{X}_{i+1}}\frac{|O_{i+1}(x)|}{|O_{i}(\phi_i(x))||\phi_i^{-1}(\phi_i(x))|}\cdot \frac{1}{|O_{i+1}(x)|}\\
&=& \frac{1}{k(\mathcal{X}_{i+1},G_{i+1})}\sum_{y\in \mathcal{X}_{i}}\frac{1}{|O_{i}(y)||\phi_i^{-1}(y)|}\cdot |\phi_i^{-1}(y)|\\
&=&\frac{k(\mathcal{X}_{i},G_{i})}{k(\mathcal{X}_{i+1},G_{i+1})},
\end{eqnarray*}
where the third equality uses the fact that $\phi_i$ is surjective.
\end{proof}

Now note that the distribution of $T_{i+1}$ in Proposition \ref{Le1} is the stationary distribution of the Burnside process for the set $\mathcal{X}_{i+1}$ and the group $G_{i+1}$. To estimate $\frac{k(\mathcal{X}_i,G_i)}{k(\mathcal{X}_{i+1},G_{i+1})}$ (where $1\leq i\leq N-1$), run this Burnside process for $B_i+N_i$ steps, where the first $B_i$ samples are burn-in samples. Suppose that the samples obtained are $\{M_{i,j}\}_{j=1}^{B_i+N_i}$. Using Proposition \ref{Le1}, estimate $\frac{k(\mathcal{X}_i,G_i)}{k(\mathcal{X}_{i+1},G_{i+1})}$ by
\begin{equation}
    \widehat{E_i}:=\frac{|G_{i+1}|}{|G_{i}|}\cdot\frac{1}{N_i}\sum_{j=B_i+1}^{B_i+N_i}\frac{|Stab_{i}(\phi_i(M_{i,j}))|}{|Stab_{i+1}(M_{i,j})||\phi_i^{-1}(\phi_i(M_{i,j}))|}.
\end{equation}
Finally, based on (\ref{EQ_1}), estimate $k(\mathcal{X},G)$ by 
\begin{equation}\label{Eq5}
    \widehat{k(\mathcal{X},G)}:=\prod_{i=1}^{N-1}\frac{1}{\widehat{E_i}}.
\end{equation}
A summary of the orbit counting algorithm is presented in Algorithm \ref{alg1}.

\begin{algorithm}
\caption{General algorithm for orbit counting}
\label{alg1}
\begin{description}
\item For $1\leq i\leq N-1$, we do the following:
\subitem Run the Burnside process for the set $\mathcal{X}_{i+1}$ and the group $G_{i+1}$ for $B_i+N_i$ steps. Suppose the samples obtained are $\{M_{i,j}\}_{j=1}^{B_i+N_i}$.
\subitem Estimate $\frac{k(\mathcal{X}_i,G_i)}{k(\mathcal{X}_{i+1},G_{i+1})}$ by 
\begin{equation*}
    \widehat{E_i}:=\frac{|G_{i+1}|}{|G_{i}|}\cdot\frac{1}{N_i}\sum_{j=B_i+1}^{B_i+N_i}\frac{|Stab_{i}(\phi_i(M_{i,j}))|}{|Stab_{i+1}(M_{i,j})||\phi_i^{-1}(\phi_i(M_{i,j}))|}.
\end{equation*}
\item Finally, estimate $k(\mathcal{X},G)$ by
\begin{equation*}
    \widehat{k(\mathcal{X},G)}:=\prod_{i=1}^{N-1}\frac{1}{\widehat{E_i}}.
\end{equation*}
\end{description}
\end{algorithm}

\paragraph{A simple example.} Consider the symmetric group $S_n$ acting on the set $\mathcal{X}_n=C_k^n=\text{all }k^n\text{ }n\text{-tuples taking values in }[k]$ (where $k\geq 2$) by permuting coordinates. For any $x\in\mathcal{X}_n$ and $l\in [k]$, define $N_l(x):=|\{j\in [n]: x_j=l\}|$. The orbits are
\begin{equation*}
    O_{i_1,\cdots,i_k}=\Big\{x\in\mathcal{X}_n: N_1(x)=i_1,\cdots,N_k(x)=i_k\Big\}
\end{equation*}
for all $(i_1,\cdots,i_k)\in \{0,1,2,\cdots\}^k$ such that $i_1+\cdots+i_k=n$. The Burnside process for this example is developed and carefully analyzed in \cite{D, DZ}. Let us use it to illustrate Algorithm \ref{alg1}.

Take $\mathcal{X}_i=C_k^i,G_i=S_i,1\leq i\leq n$ to be the sequence of sets and groups required by Algorithm \ref{alg1}. For each $1\leq i\leq n-1$, let $\phi_i:\mathcal{X}_{i+1}\rightarrow\mathcal{X}_i$ be 
\begin{equation*}
    \phi_i((x_1,x_2,\cdots,x_{i+1}))=(x_1,x_2,\cdots,x_i).
\end{equation*}
Thus $\frac{|\mathcal{X}_{i+1}|}{|\mathcal{X}_i|}=k$, $\frac{|S_{i+1}|}{|S_i|}=i+1$. For $x\in\mathcal{X}_{i+1}$, the statistic is 
\begin{equation*}
    \frac{|Stab_i(\phi_i(x))|(i+1)}{|Stab_{i+1}(x)||\phi_i^{-1}(\phi_i(x))|}.
\end{equation*}
From the definitions, 
\begin{equation*}
    |Stab_{i+1}(x)|=\prod_{l=1}^k N_l(x)!, \quad |\phi_i^{-1}(\phi_i(x))|=k,
\end{equation*}
\begin{equation*}
    |Stab_i(\phi_i(x))|= (N_{l_0}(x)-1)!\prod_{l\in [k]\backslash\{l_0\}}N_l(x)!, \quad \text{where }l_0=x_{i+1}.
\end{equation*}
Thus, the statistic is
\begin{equation*}
    \frac{i+1}{k|\{j\in [i+1]: x_j=x_{i+1}\}|}.
\end{equation*}
In \cite{D,DZ}, it is proved that for fixed $k\geq 2$, the Burnside process on $\mathcal{X}_n$, started at $(1,1,\cdots,1)$, converges in a bounded number of steps, independent of $n$.

As a reality check, we ran the counting algorithm to estimate the number of orbits for $S_n$ acting on $C_k^n$ with $n=20$ and $k\in\{2,\cdots,20\}$, setting $B_i=20$ and $N_i=10000$ for every $1\leq i\leq n-1$. The true number of orbits in this case is given by $\binom{n+k-1}{k-1}$. Figure \ref{Fig5} compares the logarithm of the true and estimated number of orbits. Additionally, Figure \ref{Fig6} presents a histogram of the logarithm of the estimated number of orbits when $k=10$, based on $100$ independent replications. These figures suggest that the counting algorithm accurately estimates the number of orbits under the current setup. The code for implementing the counting algorithm for this example is available at \url{https://github.com/cyzhong17/burnside_counting}.   
  
\begin{figure}[!h]
\centering
\includegraphics[width=0.7\textwidth]{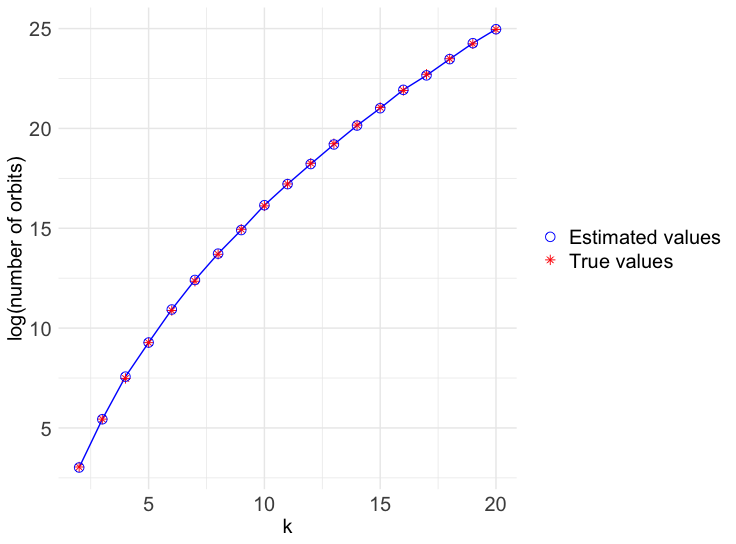}
\caption{Logarithm of true and estimated number of orbits}
\label{Fig5}
\end{figure}

\begin{figure}[!h]
\centering
\includegraphics[width=0.6\textwidth]{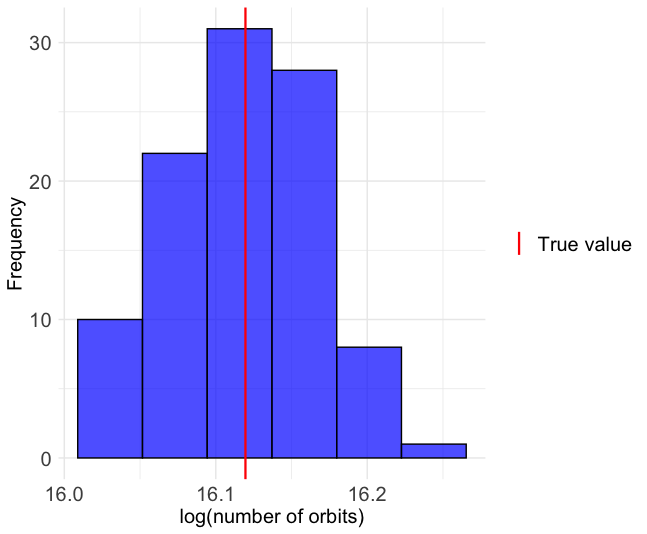}
\caption{Histogram of logarithm of estimated number of orbits for $k=10$}
\label{Fig6}
\end{figure}

\bigskip

Algorithm \ref{alg1} allows the flexibility of choosing the sequence of sets $\{\mathcal{X}_i\}_{i=1}^N$ and the sequence of groups $\{G_i\}_{i=1}^N$. Suitable choices of such sequences in concrete problems lead to efficient algorithms. Standard large deviations estimates show that for rapidly mixing Markov chains and practically bounded ratios, the estimate in (\ref{Eq5}) converges in a polynomial number of steps. See \cite{jerrum2003counting, sinclair2012algorithms} for background. There is a close connection between approximate counting and approximate sampling for a special class of problems called ``self-reducible problems'' \cite{SJ}. However, there is no known self-reducible structure for general orbit counting problems (including the counting problem discussed in Section \ref{Sect.4}). In the algorithm, the use of importance sampling is a crucial step for reducing approximate counting to approximate sampling (which is then done via the Burnside process).   

\section{Counting the number of conjugacy classes of unitriangular groups}\label{Sect.4}

This section considers the problem of approximately counting the number of conjugacy classes of the unitriangular group $U_n(\mathbb{F}_q)$. The general orbit counting algorithm is specialized to the setting where $\mathcal{X}=G=U_n(\mathbb{F}_q)$ and $G$ acts on itself by conjugation. A key idea in this specialized algorithm is to use pattern groups to construct the sequences of sets and groups. Throughout this section, for any finite group $G$, denote by $k(G)$ the number of conjugacy classes of $G$. 

A brief introduction to pattern groups is in Section \ref{Sect.4.1}. This also discusses the Burnside process for sampling uniformly from conjugacy classes of pattern groups. The counting algorithm for $U_n(\mathbb{F}_q)$ is introduced in Section \ref{Sect.4.2}. Section \ref{Sect.4.3} presents results (Theorem \ref{Thm2} and Corollary \ref{Cor4.1}) which show that the fluctuation in the importance sampling step of the algorithm is well-controlled. The proof of Theorem \ref{Thm2} is deferred to Section \ref{Sect.4.5}. Numerical results based on the algorithm are shown in Section \ref{Sect.4.4}, which demonstrate the accuracy and computational efficiency of the algorithm.

\subsection{Pattern groups and Burnside process}\label{Sect.4.1}

Pattern groups are certain subgroups of the unitriangular group $U_n(\mathbb{F}_q)$ which can be described as follows. We say that a set $J\subseteq \{(i,j): 1\leq i<j\leq n\}$ is \emph{closed} if it satisfies the property that $(i,j)\in J$ and $(j,k)\in J$ implies $(i,k)\in J$. For any closed set $J\subseteq \{(i,j): 1\leq i<j\leq n\}$, the pattern group $U_J$ is defined as
\begin{eqnarray*}
    U_J&:=&\{(u_{ij})\in U_n(\mathbb{F}_q): u_{ij}=0\text{ for every }  (i,j)\text{ such that }\nonumber\\
    &&\quad\quad\quad\quad\quad\quad\quad\quad\quad\quad 1\leq i<j\leq n\text{ and }(i,j)\notin J\}.
\end{eqnarray*}
The class of pattern groups covers many interesting subgroups of $U_n(\mathbb{F}_q)$; refer to \cite{DT} for a detailed account of pattern groups. These groups will form building blocks of the nested sequence in the counting algorithm in Section \ref{Sect.4.2}.

The Burnside process for sampling from conjugacy classes of the pattern group $U_J$ can be done by specializing the Burnside process to the setting where $\mathcal{X}=G=U_J$ and $G$ acts on itself by conjugation (see Section \ref{Sect.2.1}). For this case, the Burnside process simplifies to the following Markov chain: for each step, from $x\in U_J$, pick $y$ uniformly from the centralizer of $x$ in $U_J$ (defined as $C_{U_J}(x)=\{z\in U_J: zx=xz\}$), and move from $x$ to $y$.

For any pattern group $U_J$, the problem of picking $y$ uniformly from the centralizer $C_{U_J}(x)$ given $x\in U_J$ can be done efficiently: First note that for any $x,y\in U_J$, the condition $xy=yx$ is equivalent to 
\begin{equation*}
    (x-I_n)(y-I_n)=(y-I_n)(x-I_n),
\end{equation*}
where $I_n$ is the $n\times n$ identity matrix. Let
\begin{eqnarray*}
    \mathcal{G} &:=&\{g-I_n: g\in U_J\}\\
    &=&\{(u_{ij})\in \mathbb{F}_q^{n\times n}: u_{ij}=0 \text{ for every }(i,j)\text{ such that }\nonumber\\
    &&\quad\quad\quad\quad\quad\quad\quad\quad\quad\quad 1\leq i,j\leq n\text{ and }(i,j)  \notin J\},
\end{eqnarray*} 
and note that
\begin{equation*}
    \mathcal{V}_x:=\{y\in \mathcal{G}: y(x-I_n)=(x-I_n)y\}
\end{equation*}
is a linear space over $\mathbb{F}_q$. A basis of $\mathcal{V}_x$, denoted by $\epsilon_1,\epsilon_2,\cdots,\epsilon_d$ (where $d$ is the dimension of $\mathcal{V}_x$), can be efficiently computed using Gaussian elimination. Sample $\alpha_1,\alpha_2,\cdots,\alpha_d$ i.i.d. from the uniform distribution on $\mathbb{F}_q$. Then the random element 
\begin{equation*}
    y=I_n+\sum_{i=1}^d \alpha_i \epsilon_i
\end{equation*}
is uniformly distributed on $C_{U_J}(x)$. 

Hereafter, for any pattern group $U_J$ and any $x\in U_J$, let $O_{U_J}(x)$ denote the conjugacy class of $U_J$ that contains $x$. Note that once the dimension of $\mathcal{V}_x$ (denoted by $d$) is computed, by the orbit-stabilizer theorem, the sizes of $C_{U_J}(x)$ and $O_{U_J}(x)$ can be calculated as
\begin{equation}\label{Sizes}
    |C_{U_J}(x)|=q^d, \quad  |O_{U_J}(x)|=\frac{|U_J|}{|C_{U_J}(x)|}=q^{|J|-d}.
\end{equation}
Specializing to $U_n(\mathbb{F}_q)$, the Burnside process thus provides a method for visualizing the size distribution of conjugacy classes of $U_n(\mathbb{F}_q)$. This is achieved by collecting samples from the Burnside process, calculating the sizes of the conjugacy classes containing these samples using (\ref{Sizes}), and constructing a histogram of the conjugacy class sizes. 

\subsection{Counting algorithm}\label{Sect.4.2}

This subsection applies the general orbit counting algorithm introduced in Section \ref{Sect.3} to the problem of counting the number of conjugacy classes of unitriangular groups. The specialization involves a slight modification that helps to reduce the variance in the importance sampling step.

First set up the nested sequence: for every $(i,j)$ such that $1\leq i<j\leq n$, let $J_{i,j}$ be the set of $(k,l)$ (with $1\leq k< l\leq n$) such that either $k>i$, or $k=i$ and $l\geq j$. The resulting $N=\frac{(n-1)n}{2}$ pattern groups $\{U_{J_{i,j}}\}_{1\leq i<j\leq n}$ form a nested sequence 
\begin{equation}\label{Eq1}
    H_0\subseteq H_1 \subseteq  H_2 \subseteq \cdots \subseteq H_N=U_n(\mathbb{F}_q),
\end{equation}
where $H_0$ is the unit group consisting of the identity matrix, and $H_m=U_{J_{k_m,l_m}}$ for $m=1,2,\cdots, N$, where $k_m,l_m$ are such that 
\begin{equation*}
\frac{(n-k_m)(n-k_m-1)}{2}+n-l_m+1=m,\quad 1\leq k_m<l_m\leq n.
\end{equation*}
Note that for every $0\leq m\leq N$, $|H_m|=q^m$.

For example, when $n=4$ and $q=2$, the nested sequence of pattern groups is given by
\begin{equation*}
    \begin{bmatrix}
    1 & 0 &0 &0\\
    0&1&0 & 0\\
    0&0&1&0\\
    0&0&0&1
    \end{bmatrix},
    \begin{bmatrix}
    1 & 0 &0 &0\\
    0&1&0&0\\
    0&0&1&*\\
    0&0&0&1
    \end{bmatrix},
    \begin{bmatrix}
    1 & 0 &0 &0\\
    0&1&0&*\\
    0&0&1&*\\
    0&0&0&1
    \end{bmatrix},
    \begin{bmatrix}
    1 & 0 &0 &0\\
    0&1&* & *\\
    0&0&1&*\\
    0&0&0&1
    \end{bmatrix},
\end{equation*}
\begin{equation*}
    \begin{bmatrix}
    1 & 0 &0 &*\\
    0&1&*&*\\
    0&0&1&*\\
    0&0&0&1
    \end{bmatrix},
    \begin{bmatrix}
    1 & 0 &* &*\\
    0&1&*&*\\
    0&0&1&*\\
    0&0&0&1
    \end{bmatrix},
    \begin{bmatrix}
    1 &* &* &*\\
    0&1&*&*\\
    0&0&1&*\\
    0&0&0&1
    \end{bmatrix}.
\end{equation*}
Here, $*$ means either $0$ or $1$ can be taken in the particular position.

Next comes the importance sampling step: for every $1\leq m \leq N$, run the Burnside process on the pattern group $H_m$ (as described in Section \ref{Sect.4.1}) for $B_m+N_m$ steps. The first $B_m$ samples are burn-in samples. Suppose this gives the samples $M_{m,1}, M_{m,2},\cdots,  M_{m,B_m+N_m}$. 

Define the statistic $K_m$ on the group $H_m$ for $1\leq m\leq N$, which is a slight modification of the statistic used in the general algorithm, as follows. For every $g\in H_m$, if $g\notin H_{m-1}$, take $K_m(g)=0$; if $g\in H_{m-1}$, take $K_m(g)=\frac{|C_{H_{m-1}}(g)|}{|C_{H_m}(g)|}$. Note that both $|C_{H_{m-1}}(g)|$ and $|C_{H_m}(g)|$ can be efficiently computed based on (\ref{Sizes}). Then calculate the following estimate for $\frac{k(H_{m-1})}{k(H_m)}$:
\begin{equation}\label{Eq4}
    \hat{E}_m=\frac{q}{N_m}\sum_{j=B_m+1}^{B_m+N_m} K_m(M_{m,j}).
\end{equation}
Finally, use the following estimate of $k(U_n(\mathbb{F}_q))$:
\begin{equation}
    \widehat{k(U_n(\mathbb{F}_q))}=\frac{1}{\prod_{m=1}^N \hat{E}_m}.
\end{equation}

The validity of the algorithm is justified by the following proposition. Note that the distribution of $T_m$ is the stationary distribution of the Burnside process for sampling from conjugacy classes of $H_m$.

\begin{proposition}\label{Pro4.1}
Suppose that $1\leq m\leq N$. Let $T_m$ be a random element of the group $H_m$ such that
\begin{equation*}
    \mathbb{P}(T_m=g) =\frac{1}{k(H_m)} \frac{1}{|O_{H_m}(g)|}, \quad \text{for all } g\in H_m.
\end{equation*}
Then
\begin{equation*}
    q\mathbb{E}[K_m(T_m)]=\frac{k(H_{m-1})}{k(H_m)}.
\end{equation*}
\end{proposition}
\begin{proof}
By the orbit-stabilizer theorem, for any $x\in H_{m-1}$, 
\begin{equation}\label{orbitstab}
    |H_m|=|O_{H_m}(x)||C_{H_m}(x)|, \quad |H_{m-1}|=|O_{H_{m-1}}(x)||C_{H_{m-1}}(x)|.
\end{equation}
Hence
\begin{eqnarray*}
 q\mathbb{E}[K_m(T_m)]&=&
 \frac{q}{k(H_m)}\sum_{x\in H_{m-1}}\frac{|C_{H_{m-1}}(x)|}{|C_{H_m}(x)|}\frac{1}{|O_{H_m}(x)|}\\
 &=& \frac{1}{k(H_m)}\sum_{x\in H_{m-1}}\frac{1}{|O_{H_{m-1}}(x)|} = \frac{k(H_{m-1})}{k(H_m)}.
\end{eqnarray*}
\end{proof}

\subsection{Controlling the fluctuation for the importance sampling step}\label{Sect.4.3}

In importance sampling, it is important to control the fluctuation of the resulting estimate so that it gives a close approximation of the desired expectation \cite{Liu, CD}. For each pair of adjacent pattern groups $H_{m-1},H_m$ (where $1\leq m\leq N$) in the nested sequence (\ref{Eq1}), the following results demonstrate that the variance in the importance sampling step is well-controlled.
\begin{theorem}\label{Thm2}
For all $1\leq m\leq N$ and any prime power $q$,
\begin{equation*}
    q^{-1}\leq \frac{k(H_m)}{k(H_{m-1})}\leq q^3.
\end{equation*}
\end{theorem}
\begin{corollary}\label{Cor4.1}
For all $1\leq m\leq N$ and any prime power $q$, 
\begin{equation*}
     \sqrt{\mathrm{Var}(qK_m(T_m))} \leq q^2\cdot\mathbb{E}[qK_m(T_m)]=q^2\cdot\frac{k(H_{m-1})}{k(H_m)},
\end{equation*}
where $T_m$ is defined as in Proposition \ref{Pro4.1}.
\end{corollary}
\begin{remark}
Note that for each $1\leq m\leq N$, $qK_m(T_m)$ is the statistic used in the importance sampling step for the pair of pattern groups $H_{m-1},H_m$. For the application here, $q$ is fixed and $n$ can be large. Corollary \ref{Cor4.1} thus implies that the standard deviation of the statistic $qK_m(T_m)$ is at most of the same order as the target ratio $\frac{k(H_{m-1})}{k(H_m)}$. Therefore, the variance for the importance sampling step of the counting algorithm is well-controlled. 
\end{remark}
\begin{proof}[Proof of Corollary \ref{Cor4.1} (based on Theorem \ref{Thm2})]
Note that for any $g\in H_{m-1}$, $|C_{H_{m-1}}(g)|\leq |C_{H_m}(g)|$. Hence $K_m(g)\in [0,1]$ for any $g\in H_m$. By Proposition \ref{Pro4.1},  
\begin{equation*}
     \mathbb{E}[(K_m(T_m))^2]\leq \mathbb{E}[K_m(T_m)]=\frac{k(H_{m-1})}{q k(H_m)}.
\end{equation*}
Hence by Theorem \ref{Thm2},
\begin{eqnarray*}
    \sqrt{\mathrm{Var}(qK_m(T_m))}&\leq& q\sqrt{\mathbb{E}[(K_m(T_m))^2]}\leq \sqrt{\frac{q k(H_{m-1})}{k(H_m)}}\nonumber\\
    &\leq& q^2\cdot\frac{k(H_{m-1})}{k(H_m)}=q^2\cdot\mathbb{E}[qK_m(T_m)].
\end{eqnarray*}
\end{proof}

The proof of Theorem \ref{Thm2} will be presented in Section \ref{Sect.4.5}. Interestingly, the proof of this deterministic result relies on a probabilistic interpretation of the Burnside process for sampling from conjugacy classes of pattern groups. 

\subsection{Numerical results}\label{Sect.4.4}

This subsection presents some numerical results based on the implementation of the counting algorithm. The code is available at \url{https://github.com/cyzhong17/burnside_counting}. Concretely, we take $q=2, 3$ and $n=1,2, \cdots, 32$, and implement this algorithm to estimate $k(U_n(\mathbb{F}_q))$. When estimating the ratio $\frac{k(H_{m-1})}{k(H_m)}$ based on (\ref{Eq4}), take $B_m=100000$ (the number of burn-in samples) and $N_m\in \{100000,200000,300000\}$ (the number of samples used for importance sampling). The starting state of the Burnside process is taken to be the identity matrix. 

For $q=2$, the estimated values of $\log_2(k(U_n(\mathbb{F}_2)))$ and $\log_2(k(U_n(\mathbb{F}_2)))\slash (n^2)$ are plotted in Figures \ref{Fig1} and \ref{Fig2} for the three choices of $N_m$. Similarly, for $q=3$, the estimated values of $\log_3(k(U_n(\mathbb{F}_3)))$ and $\log_3(k(U_n(\mathbb{F}_3)))\slash (n^2)$ are presented in Figures \ref{Fig3} and \ref{Fig4}. For reference, these plots indicate the true values of $\log_q(k(U_n(\mathbb{F}_q)))$ and $\log_q(k(U_n(\mathbb{F}_q)))\slash (n^2)$ for $n=1,2,\cdots,16$ and the corresponding $q\in\{2,3\}$ (based on \cite{pak2015higman,soffer2016combinatorics}). Observe that the estimated values for different $N_m$ match closely with each other, which implies that the chosen sample sizes are sufficient for giving reliable estimates. For $n=1,2,\cdots,16$, the estimated values are very close to the true values, which demonstrates the accuracy of our algorithm.

In Figures \ref{Fig2} and \ref{Fig4}, observe that $\log_q(k(U_n(\mathbb{F}_q)))\slash (n^2)$ for both $q\in\{2,3\}$ appears to approach $1\slash 12$ as $n$ increases. This provides numerical evidence for the conjecture that the lower bound in (\ref{Bound_Higman}) is sharp (\cite{soffer2016upper, soffer2016combinatorics}). A refinement of Higman's conjecture (\cite{vera2003conjugacy} and \cite[Section 10.4]{soffer2016combinatorics}) postulates that for every $n\in\mathbb{N}^{*}$, $k(U_n(\mathbb{F}_q))$ is a polynomial in $q$ of degree $\lfloor n(n+6)\slash 12\rfloor$. In light of this conjecture, the values
\begin{equation*}
    \frac{n(n+6)}{12 n^2}=\frac{n+6}{12 n}
\end{equation*}
are also indicated in Figures \ref{Fig2} and \ref{Fig4}.

\begin{figure}[!h]
\centering
\includegraphics[width=\textwidth]{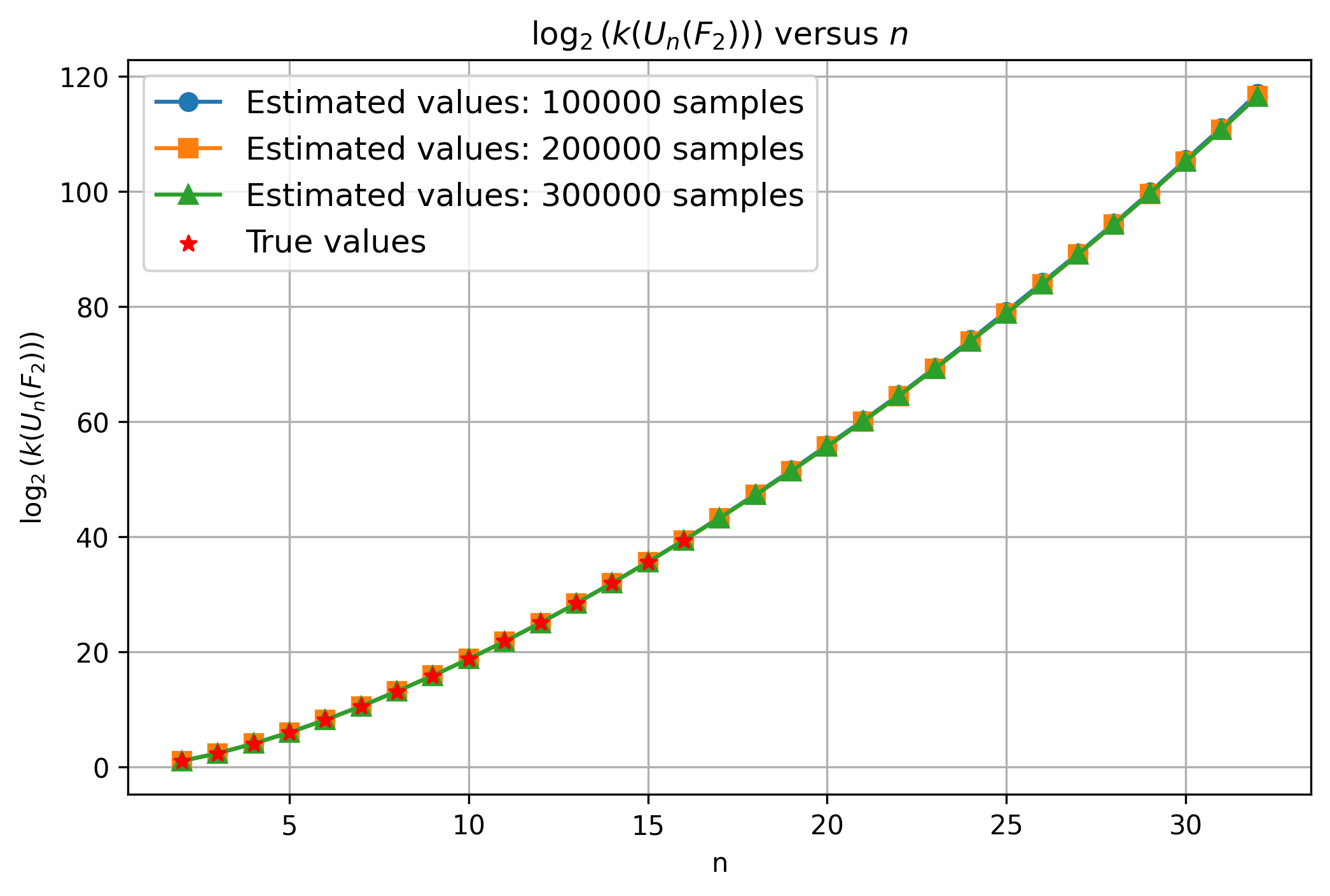}
\caption{Plot of $\log_2(k(U_n(\mathbb{F}_2)))$ for $n=1,2,\cdots,32$}
\label{Fig1}
\end{figure}

\begin{figure}[!h]
\centering
\includegraphics[width=\textwidth]{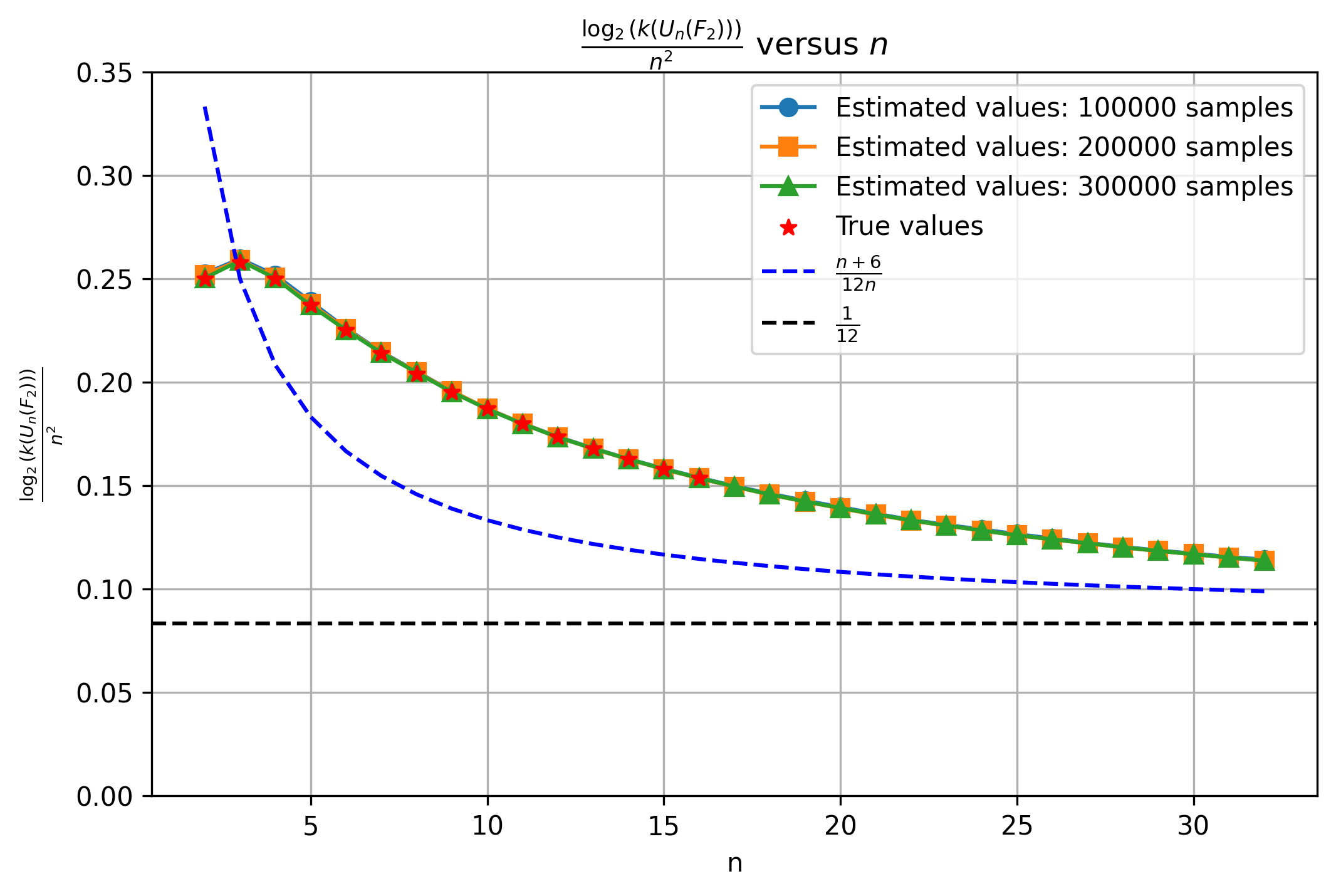}
\caption{Plot of $\log_2(k(U_n(\mathbb{F}_2)))\slash (n^2)$ for $n=1,2,\cdots,32$}
\label{Fig2}
\end{figure}

\begin{figure}[!h]
\centering
\includegraphics[width=\textwidth]{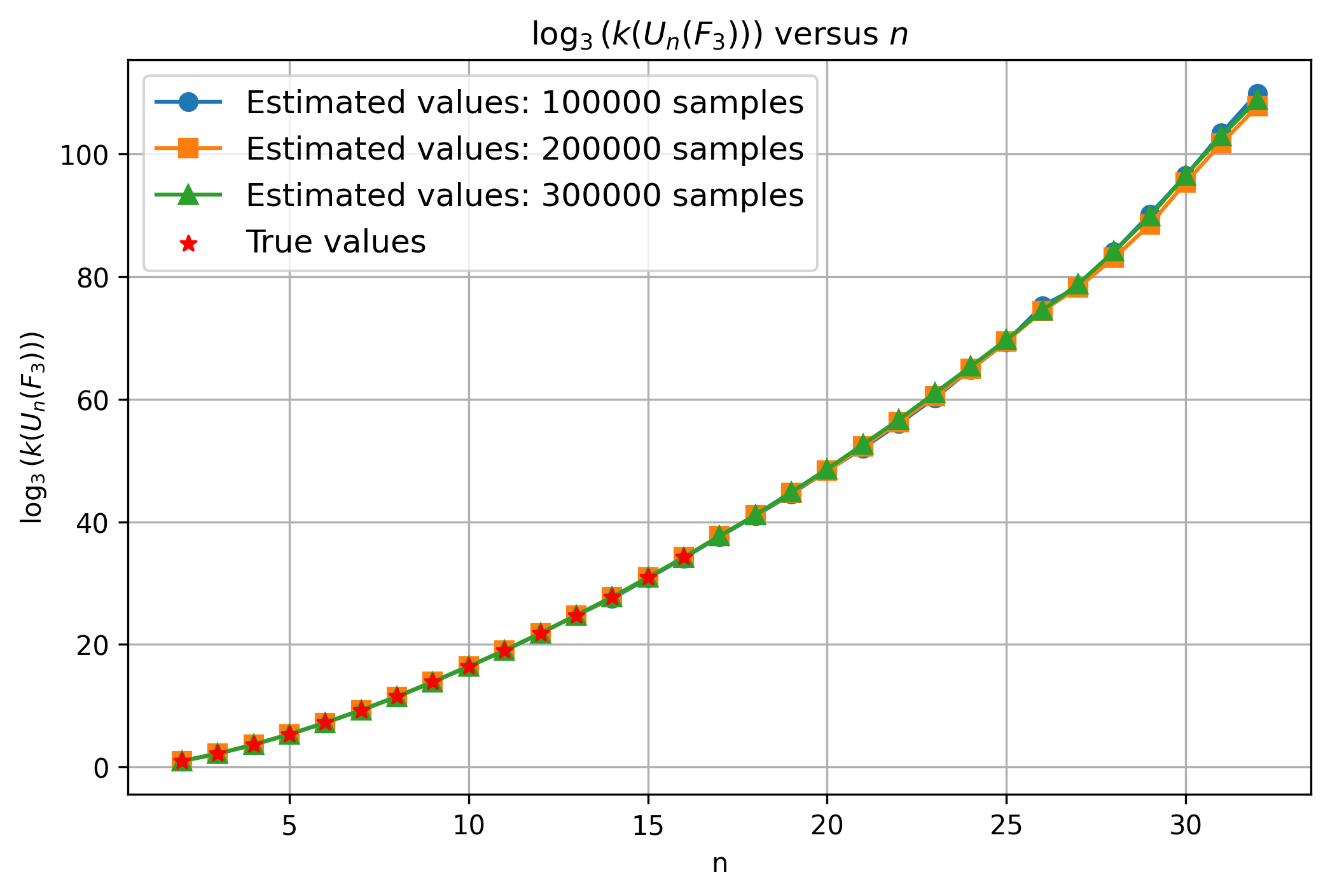}
\caption{Plot of $\log_3(k(U_n(\mathbb{F}_3)))$ for $n=1,2,\cdots,32$}
\label{Fig3}
\end{figure}

\begin{figure}[!h]
\centering
\includegraphics[width=\textwidth]{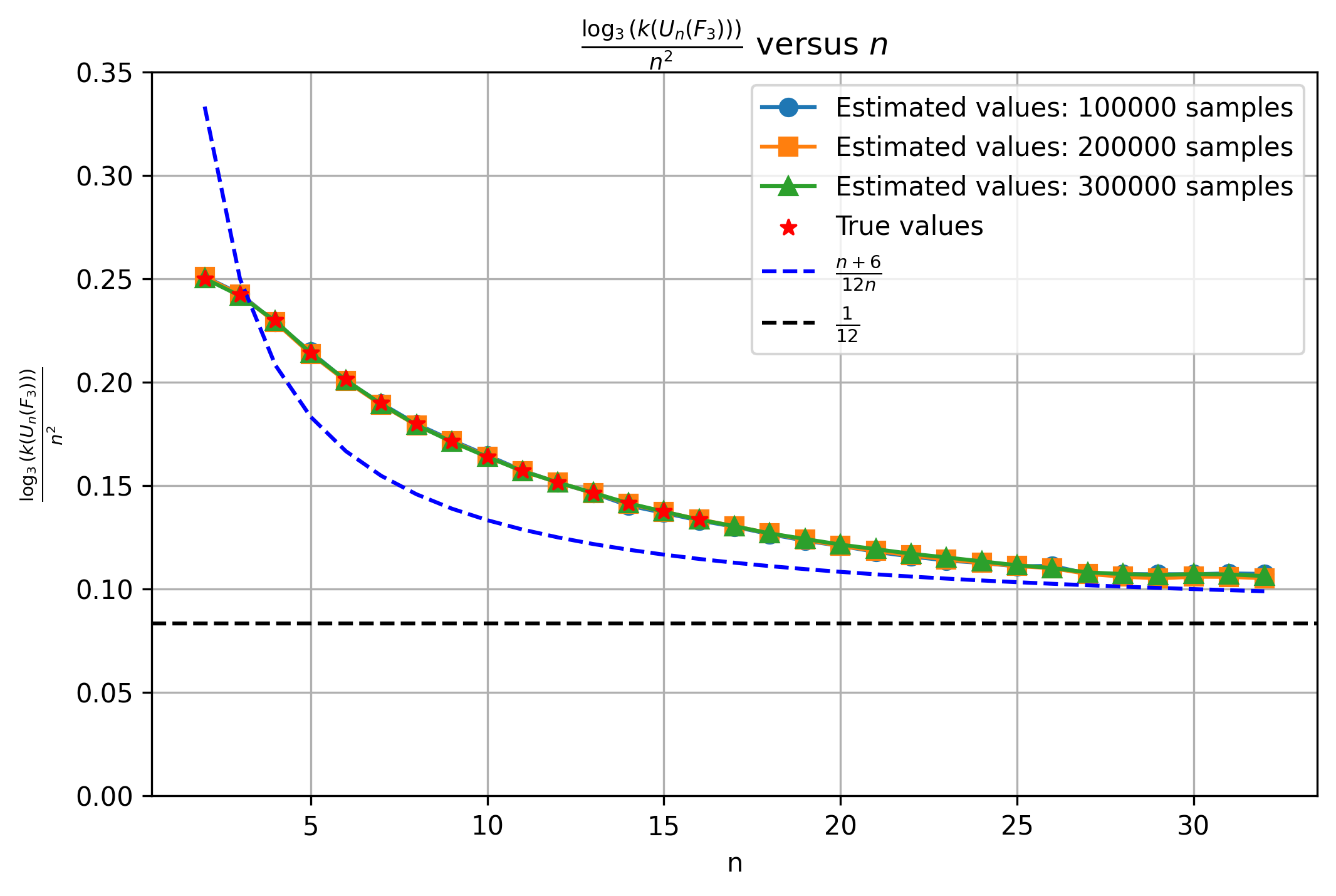}
\caption{Plot of $\log_3(k(U_n(\mathbb{F}_3)))\slash (n^2)$ for $n=1,2,\cdots,32$}
\label{Fig4}
\end{figure}

\subsection{Proof of Theorem \ref{Thm2}}\label{Sect.4.5}

This subsection gives the proof of Theorem \ref{Thm2} based on a probabilistic interpretation of the Burnside process for sampling from conjugacy classes of pattern groups (as described in Section \ref{Sect.4.1}). We focus on the case where $m=N$; the proof for other cases is similar and omitted. The result holds automatically for $n\leq 2$; henceforth, we assume that $n\geq 3$ throughout the remainder of the proof.

Below, take $G=H_N=U_n(\mathbb{F}_q)$ and $H=H_{N-1}$. By Burnside's lemma, 
\begin{equation}\label{Eq2}
    k(G)=\frac{1}{|G|}\sum_{g\in G}|C_{G}(g)|.
\end{equation}
Consider the following probabilistic interpretation of (\ref{Eq2}): let $A$ be sampled from $G$ uniformly at random, then 
\begin{equation}\label{BurnsideLemma}
    k(G)=\mathbb{E}[|\{X\in G: AX=XA\}|].
\end{equation}
Given $A$, the quantity $|\{X\in G: AX=XA\}|$ can be represented based on a probabilistic interpretation of the Burnside process. A similar interpretation can be provided for $k(H)$. 

\begin{proof}[Proof of Theorem \ref{Thm2}]

The proof proceeds by first introducing an algorithm that generates $A,B\in G$ and $\phi(A),\tilde{B}\in H$ such that $A$ is uniformly distributed on $G$, $\phi(A)$ is uniformly distributed on $H$, $B$ is uniformly distributed on the centralizer of $A$ in $G$, and $\tilde{B}$ is uniformly distributed on the centralizer of $\phi(A)$ in $H$. Then, based on the aforementioned probabilistic interpretation of Burnside's lemma, we represent both $k(G)$ and $k(H)$ using quantities involved in the algorithm. This representation allows us to compare $k(G)$ and $k(H)$ and reach the desired bound.

Let random matrices $A,B\in G$ and $\tilde{B}\in H$ be sampled via the algorithm to follow, where $B$ and $\tilde{B}$ only differ in the first row. Let $\phi(A)\in H$ be obtained from $A$ by setting the $(1,2)$ entry of $A$ to $0$. Represent $A,B,\tilde{B}$ as
\begin{equation*}
    A=\begin{bmatrix}
   1 & a_{1,2} & a_{1,3} & \cdots & a_{1,n-1} & a_{1,n}\\
    0& 1 & a_{2,3} & \cdots & a_{2,n-1} & a_{2,n}\\
  &&& \cdots & & \\
    0 & 0 & 0 & \cdots & 1 & a_{n-1,n}\\
    0 & 0 & 0 & \cdots & 0 & 1
    \end{bmatrix}, 
\end{equation*}
\begin{equation*}
    B=\begin{bmatrix}
   1 & b_{1,2} & b_{1,3} & \cdots & b_{1,n-1} & b_{1,n}\\
    0& 1 & b_{2,3} & \cdots & b_{2,n-1} & b_{2,n}\\
  & & &  \cdots & &  \\
    0 & 0 & 0 & \cdots & 1 & b_{n-1,n}\\
    0 & 0 & 0 & \cdots & 0 & 1
    \end{bmatrix},
\end{equation*}
\begin{equation*}
    \tilde{B}=\begin{bmatrix}
   1 & 0 & \tilde{b}_{1,3} & \cdots & \tilde{b}_{1,n-1} & \tilde{b}_{1,n}\\
    0& 1 & b_{2,3} & \cdots & b_{2,n-1} & b_{2,n}\\
  & & &  \cdots & &  \\
    0 & 0 & 0 & \cdots & 1 & b_{n-1,n}\\
    0 & 0 & 0 & \cdots & 0 & 1
    \end{bmatrix}.
\end{equation*}
Note that $AB=BA$ if and only if $\sum_{i<k<j} a_{i,k} b_{k,j}=\sum_{i<k<j} b_{i,k} a_{k,j}$ for all $1\leq i<j\leq n$. This condition is 
equivalent to requiring that for all $1\leq i\leq n-2$,
\begin{equation}\label{commutingcondition}
    \sum_{l=i+1}^{n-1} a_{i, l}\begin{bmatrix}
    0\\
    \cdots\\
    0\\
    b_{l,l+1}\\
    \cdots\\
    b_{l,n}
    \end{bmatrix}
  = \sum_{l=i+1}^{n-1} b_{i, l}\begin{bmatrix}
    0\\
    \cdots\\
    0\\
    a_{l,l+1}\\
    \cdots\\
    a_{l,n}
    \end{bmatrix}, 
\end{equation}
where both sides are in $\mathbb{F}_q^{n-i-1}$. Moreover, $\phi(A)\tilde{B}=\tilde{B}\phi(A)$ if and only if (\ref{commutingcondition}) holds for all $2\leq i\leq n-2$ and 
\begin{equation}\label{commutingcondition1}
     \sum_{l=3}^{n-1} a_{1, l}\begin{bmatrix}
    0\\
    \cdots\\
    0\\
    b_{l,l+1}\\
    \cdots\\
    b_{l,n}
    \end{bmatrix}
  = \sum_{l=3}^{n-1} \tilde{b}_{1, l}\begin{bmatrix}
    0\\
    \cdots\\
    0\\
    a_{l,l+1}\\
    \cdots\\
    a_{l,n}
    \end{bmatrix},
\end{equation}
where both sides are in $\mathbb{F}_q^{n-3}$. For every $1\leq i\leq n-2$, define
\begin{equation*}
    A_i:=\begin{bmatrix}
    a_{i+1,i+2} & a_{i+1,i+3} & \cdots & a_{i+1,n}\\
    0& a_{i+2,i+3} & \cdots & a_{i+2,n}\\
    &&\cdots & \\
    0 & 0 & \cdots & a_{n-1,n}
    \end{bmatrix},
\end{equation*}
\begin{equation*}
    B_i:=\begin{bmatrix}
    b_{i+1,i+2} & b_{i+1,i+3} & \cdots & b_{i+1,n}\\
    0& b_{i+2,i+3} & \cdots & b_{i+2,n}\\
    &&\cdots & \\
    0 & 0 & \cdots & b_{n-1,n}
    \end{bmatrix}.
\end{equation*}

The algorithm sequentially generates the $i$th rows of $A$ and $B$ for $i=n-1,\cdots,2,1$. For $i=n-1$, independently sample $a_{n-1,n}$ and $b_{n-1,n}$ from $\mathrm{Unif}(\mathbb{F}_q)$, the uniform distribution on $\mathbb{F}_q$. For $i=n-2,\cdots,2$, in updating the $i$th rows of $A$ and $B$ given $A_i$ and $B_i$ (assuming no rejection has occurred previously), sample $a_{i,i+1},\cdots,a_{i,n}$ i.i.d. from $\mathrm{Unif}(\mathbb{F}_q)$. If 
\begin{equation}\label{acceptance}
    \sum_{l=i+1}^{n-1} a_{i, l}\begin{bmatrix}
    0\\
    \cdots\\
    0\\
    b_{l,l+1}\\
    \cdots\\
    b_{l,n}
    \end{bmatrix}
   \in   \mathrm{Span}\left\{\begin{bmatrix}
    a_{i+1,i+2}\\
    a_{i+1,i+3}\\
    \cdots\\
    a_{i+1,n}
    \end{bmatrix},
    \cdots,
  \begin{bmatrix}
    0\\
    \cdots\\
    0\\
    a_{n-1,n}
    \end{bmatrix}
   \right\},
\end{equation}
the algorithm accepts and chooses $(b_{i,i+1},\cdots,b_{i,n})^{\top}$ uniformly from the elements in $\mathbb{F}_q^{n-i}$ that satisfy (\ref{commutingcondition}); otherwise it rejects. Note that in the former case, there are $q^{n-i-\mathrm{rank}(A_i)}$ possible choices. For $i=1$, assuming no rejection has occurred previously, first sample $a_{1,2},\cdots,a_{1,n}$ i.i.d. from $\mathrm{Unif}(\mathbb{F}_q)$. Then there are two versions: for version 1, if (\ref{acceptance}) holds with $i=1$, the algorithm accepts and chooses $(b_{1,2},\cdots,b_{1,n})^{\top}$ uniformly from the elements in $\mathbb{F}_q^{n-1}$ that satisfy (\ref{commutingcondition}) with $i=1$ (there are $q^{n-1-\mathrm{rank}(A_1)}$ possible choices); otherwise it rejects. For version $2$, if
\begin{equation}
    \sum_{l=3}^{n-1} a_{1, l}\begin{bmatrix}
    0\\
    \cdots\\
    0\\
    b_{l,l+1}\\
    \cdots\\
    b_{l,n}
    \end{bmatrix}
   \in \mathrm{Span}\left\{\begin{bmatrix}
    a_{3,4}\\
    a_{3,5}\\
    \cdots\\
    a_{3,n}
    \end{bmatrix},
    \cdots,
  \begin{bmatrix}
    0\\
    \cdots\\
    0\\
    a_{n-1,n}
    \end{bmatrix}
   \right\},
\end{equation}
the algorithm accepts and chooses $(\tilde{b}_{1,3},\cdots,\tilde{b}_{1,n})^{\top}$ uniformly from the elements in $\mathbb{F}_q^{n-2}$ that satisfy (\ref{commutingcondition1}) (there are $q^{n-2-\mathrm{rank}(A_2)}$ possible choices), independent of $(b_{1,2},\cdots,b_{1,n})^{\top}$; otherwise it rejects. Once the algorithm rejects, it samples the remaining entries of $A$ that are strictly above the diagonal i.i.d. from $\mathrm{Unif}(\mathbb{F}_q)$, and outputs $A$ without $B$.

Note that $A$ is uniformly distributed on $G$ and $\phi(A)$ is uniformly distributed on $H$. Moreover, conditional on $A$ and the event of no rejection for version 1, the distribution of $B$ is the same as that of the state obtained by running the Burnside process on $G$ for one iteration starting from $A$ (i.e., $B$ is uniformly distributed on the centralizer of $A$ in $G$). Similarly, conditional on $A$ and the event of no rejection for version 2, the distribution of $\tilde{B}$ is the same as that of the state obtained by running the Burnside process on $H$ for one iteration starting from $\phi(A)$ (i.e., $\tilde{B}$ is uniformly distributed on the centralizer of $\phi(A)$ in $H$).

Note that, according to the probabilistic interpretation of Burnside's lemma (see (\ref{BurnsideLemma})),
\begin{eqnarray*}
    && k(G)=\mathbb{E}[|\{X\in G: AX=XA\}|]\\
    &=&\mathbb{E}\big[q^{1+\sum_{i=1}^{n-2}(n-i-\mathrm{rank}(A_i))}\mathbb{P}(\text{no rejection for version 1}|A)\big],
\end{eqnarray*}
\begin{eqnarray*}
   && k(H)=\mathbb{E}[|\{X\in H: \phi(A)X=X\phi(A)\}|]\\
   &=&\mathbb{E}\big[q^{1+\sum_{i=2}^{n-2}(n-i-\mathrm{rank}(A_i))+(n-2-\mathrm{rank}(A_2))}\mathbb{P}(\text{no rejection for version 2}|A) \big]. 
\end{eqnarray*}
Let 
\begin{equation*}
    f(A_1):=1+\sum_{i=1}^{n-2}(n-i-\mathrm{rank}(A_i)),
\end{equation*}
\begin{equation*}
    g(A_1):=1+\sum_{i=2}^{n-2}(n-i-\mathrm{rank}(A_i))+(n-2-\mathrm{rank}(A_2)).
\end{equation*} 
Note that
\begin{equation}\label{Eq1.1}
    f(A_1)-1\leq g(A_1)\leq f(A_1).
\end{equation}
Denote by $\mathcal{Y}$ the set of $(n-2)\times (n-2)$ upper triangular matrices over $\mathbb{F}_q$. Then
\begin{eqnarray*}
k(G)&=& \sum_{Y\in\mathcal{Y}}q^{f(Y)}\mathbb{E}[\mathbbm{1}_{A_1=Y}\mathbb{P}(\text{no rejection for version 1}|A)]\\
&=& \sum_{Y\in\mathcal{Y}}q^{f(Y)}\mathbb{E}[\mathbb{P}(\{\text{no rejection for version 1}\}\cap\{A_1=Y\}|A)]\\
&=& \sum_{Y\in\mathcal{Y}}q^{f(Y)}\mathbb{E}[\mathbb{P}(\text{no rejection for version 1}|A_1,B_1)\mathbbm{1}_{A_1=Y}].
\end{eqnarray*}
Similarly,
\begin{eqnarray*}
k(H)&=& \sum_{Y\in \mathcal{Y}}q^{g(Y)}\mathbb{E}[\mathbbm{1}_{A_1=Y}\mathbb{P}(\text{no rejection for version 2}|A)]\\
&=& \sum_{Y\in\mathcal{Y}}q^{g(Y)}\mathbb{E}[\mathbb{P}(\text{no rejection for version 2}|A_1,B_1)\mathbbm{1}_{A_1=Y}].
\end{eqnarray*}

Before sampling the first rows of $A,B,\tilde{B}$, either rejection occurs for both versions or for neither version. In the following, we assume the latter case, and condition on $A_1,B_1$. Below, for each $2\leq i\leq n-1$, write
\begin{equation*}
    \eta_i'=(0,\cdots,0,a_{i,i+1},\cdots,a_{i,n})^{\top}\in\mathbb{R}^{n-2},
\end{equation*}
\begin{equation*}
    \eta_i=(0,\cdots,0,b_{i,i+1},\cdots,b_{i,n})^{\top}\in\mathbb{R}^{n-2}.
\end{equation*}
The event of no rejection for version 1 is equivalent to 
\begin{equation*}
    \sum_{l=2}^{n-1}a_{1,l}\eta_l\in \mathrm{Span}\{\eta_2',\cdots,\eta_{n-1}'\},
\end{equation*}
and the event of no rejection for version 2 is equivalent to 
\begin{equation*}
    \sum_{l=3}^{n-1}a_{1,l}\eta_l\in \mathrm{Span}\{\eta_3',\cdots,\eta_{n-1}'\}.
\end{equation*}

Let 
\begin{equation*}
    U_1=\mathrm{Span}\{\eta_2,\cdots,\eta_{n-1}\}, \quad U_2=\mathrm{Span}\{\eta_3,\cdots,\eta_{n-1}\};
\end{equation*}
\begin{equation*}
    W_1=\mathrm{Span}\{\eta_2',\cdots,\eta_{n-1}'\},\quad W_2=\mathrm{Span}\{\eta_3',\cdots,\eta_{n-1}'\}.
\end{equation*}
Now
\begin{equation*}
    \mathrm{dim}(U_1\slash U_1\cap W_1)=\mathrm{dim}(U_1)-\mathrm{dim}(U_1\cap W_1),
\end{equation*}
\begin{equation*}
    \mathrm{dim}(U_2\slash U_2\cap W_2)=\mathrm{dim}(U_2)-\mathrm{dim}(U_2\cap W_2).
\end{equation*}
Hence by the fact that $U_2\cap W_2\subseteq U_1\cap W_1$, 
\begin{equation}
    \mathrm{dim}(U_1\slash U_1\cap W_1)-\mathrm{dim}(U_2\slash U_2\cap W_2)\leq \mathrm{dim}(U_1)-\mathrm{dim}(U_2)\leq 1.
\end{equation}
To bound $\mathrm{dim}(U_1\cap W_1\slash U_2\cap W_1)$, suppose there are two linearly independent elements $\bar{u},\bar{v}\in U_1\cap W_1\slash U_2\cap W_1$. Let $u=a_1\eta_2+u_1$, $v=a_2\eta_2+u_2$, where $u_1,u_2\in U_2$. If $a_1=0$, then $u=u_1\in U_2\cap W_1$, a contradiction to the fact that $\bar{u},\bar{v}$ are linearly independent. Hence $a_1\neq 0$. Similarly $a_2\neq 0$. Now $a_2u-a_1v=a_2u_1-a_1u_2\in  U_2\cap W_1$, hence $a_2\bar{u}-a_1\bar{v}=\bar{0}$. This again leads to a contradiction. Therefore, $\mathrm{dim}(U_1\cap W_1\slash U_2\cap W_1)\leq 1$. A similar argument shows that $\mathrm{dim}(U_2\cap W_1\slash U_2\cap W_2)\leq 1$. Thus 
\begin{equation}
    \mathrm{dim}(U_1\slash U_1\cap W_1)-\mathrm{dim}(U_2\slash U_2\cap W_2)\geq \mathrm{dim}(U_2\cap W_2)-\mathrm{dim}(U_1\cap W_1)\geq -2.
\end{equation}
Now 
\begin{equation*}
    \mathbb{P}(\text{no rejection for version 1}|A_1,B_1)=q^{-\mathrm{dim}(U_1\slash U_1\cap W_1)},
\end{equation*}
\begin{equation*}
    \mathbb{P}(\text{no rejection for version 2}|A_1,B_1)=q^{-\mathrm{dim}(U_2\slash U_2\cap W_2)}.
\end{equation*}
Hence we conclude that
\begin{eqnarray*}
   &&q^{-2}\mathbb{P}(\text{no rejection for version 1}|A_1,B_1)\\
   &\leq& \mathbb{P}(\text{no rejection for version 2}|A_1,B_1)\\
   &\leq& 
   q\mathbb{P}(\text{no rejection for version 1}|A_1,B_1).
\end{eqnarray*}
Finally, from (\ref{Eq1.1}), it follows that
\begin{equation*}
 q^{-3}k(G)    \leq k(H) \leq q k(G).
\end{equation*}
\end{proof}

\section{Final remarks}\label{Sect.5}

The approach of this paper seems broadly useful for the many applications of the Burnside process: partitions \cite{diaconis2024poisson}, trees \cite{bartholdi2024algorithm}, and contingency tables \cite{dittmer2019counting, diaconis2022statistical} should all work well. 

There are many further approaches to estimating totals given a sample: capture-recapture \cite{pollock2000capture}, Good-Turing estimates for the number of species \cite{good1953population}, and Bayesian approaches \cite{bunge1993estimating, solow1994bayesian}. We hope to try these in future work.

An important theoretical problem is to find useful rates of convergence for the Burnside process in this application. For preliminary efforts, see \cite{diaconis2024poisson, zhong2024}.

\section{Acknowledgments and funding sources}

Acknowledgments: We thank Susan Holmes for helpful suggestions. We also thank Andrew Lin and Nathan Tung for helpful comments. 

\noindent Funding: This work was supported by the National Science Foundation [grant number 1954042].

\bibliographystyle{acm}
\bibliography{Burnside.bib}
\end{document}